\documentclass{article}

\usepackage{lineno,hyperref}
\usepackage{amssymb,amsmath, amsfonts, amsthm}

\usepackage{graphicx}
\usepackage{pgfplots}
\usepackage{tikz}
\usepackage{booktabs}
\usepackage{array}
\usepackage{verbatim}
\usepackage{subfig}
\usepackage{geometry}
\usepackage[makeroom]{cancel}
\usepackage{fancyhdr}

\newcommand{\vertex}{\node[vertex]}
\tikzstyle{vertex}=[circle, draw, inner sep=0pt, minimum size=6pt]
\newtheorem{theorem}{Theorem}

\newtheorem{corollary}{Corollary}

\modulolinenumbers[5]

\begin{document}

\title{Diamond-free, claw-free cubic graphs are $(1, 1, 2, 3)$-packing colorable}
\author{Sarah E. Anderson$^a$ \and Kirsti Kuenzel$^b$ \and Juan D. Marcano Cuellar$^b$}

\date{}

\maketitle

\begin{center}
$^a$ Department of Mathematics, University of St. Thomas, St. Paul, Minnesota, USA\\
$^b$ Department of Mathematics, Trinity College, Hartford, CT, USA\\
\end{center}
\vskip15mm

\begin{abstract}
A $(1, 1, 2, k)$-packing coloring of a graph $G$ is a partition of $V(G)$ into two independent sets, a 2-packing, and a $k$-packing. Recently, the question was posed in 
[A short proof that every claw-free cubic graph is (1, 1, 2, 2)-packing colorable, arXiv:2512.24001v1] as to whether every claw-free cubic graph is $(1, 1, 2, 3)$-packing colorable. We provide an answer in the affirmative in the case that $G$ is a diamond-free, claw-free cubic graph.

\end{abstract}

\section{Introduction} \label{sec:intro}
Given a non-decreasing sequence $S= (a_1, \dots, a_r)$ where each $a_i$ is a positive integer for $i \in [r]$, a function $f:V(G) \to [r]$ is an \emph{$S$-packing coloring} of $G$ if for every two vertices $u, v \in V(G)$ with $f(u) = f(v)$, the distance between $u$ and $v$ is greater than $a_i$. Gastineau and Togni \cite{gt-2016} posed the question of whether the subdivision of every subcubic graph admits a $(1, 2, 3, 4, 5)$-packing coloring. It was later shown that this is equivalent to asking whether every subcubic graph, other than the Petersen graph, is $(1, 1, 2, 2)$-packing colorable. Many papers provide partial results in an effort to answer this question. Some recent work has focused on $i$-saturated subcubic graphs, which is a way to classify graphs based on adjacencies among degree three vertices. A subcubic graph $G$ is called $2$-saturated if every vertex of degree $3$ in $G$ is adjacent to at most two other vertices of degree $3$. Mortada and Togni \cite{mt-20242} proved that every $2$-saturated subcubic graph is $(1, 1, 2, 3)$-packing colorable. A vertex of degree $3$ is heavy if all the vertices it is adjacent to have degree $3$. A subcubic graph $G$ is called $(3,0)$-saturated graph if every heavy vertex is not adjacent to another heavy vertex. Recently, it was proven in \cite{mz-2026} that every $(3, 0)$-saturated subcubic graph is $(1, 1, 2, 3)$-packing colorable. Additionally, Yang and Wu \cite{YW-2023} proved that every 3-irregular subcubic graph is $(1, 1, 3)$-packing colorable. Tarhini and Togni  \cite{tt-2024} showed that every cubic Halin graph is $(1, 1, 2, 3)$-packing colorable. Liu et al. \cite{liu-2020} proved that subcubic graphs with maximum average degree less that $\frac{30}{11}$ are ($1, 1, 2, 2)$-packing colorable. A major step forward was made by El Zein and Mortada in \cite{EM-2025} that every nonregular subcubic graph is $(1, 1, 2, 2)$-colorable. All that remains is to show that all cubic graphs, other than the Petersen graph, are indeed $(1, 1, 2, 2)$-colorable. 

In \cite{BKR-2024} Bre\v{s}ar et al. proved that every claw-free cubic graph is $(1, 1, 2, 2)$-packing colorable. Then Mortada and El Zein \cite{MZ-2026a} provided an easier proof that all claw-free cubic graphs are $(1, 1, 2, 2$)-packing colorable. In that same paper, they posed the question of whether every claw-free cubic graph is $(1, 1, 2, 3)$-packing colorable. In this  paper, we prove that every diamond-free, claw-free cubic graph is indeed $(1, 1, 2, 3)$-packing colorable.

\subsection{Definitions and Known Results}
Given a graph $G = (V(G), E(G))$, we write $n(G)$ to denote $|V(G)|$. For any $v \in V(G)$, the \emph{open neighborhood} of $v$ is defined to be $N_G(v) = \{u: uv \in E(G)\}$ (or simply $N(v)$ when the context is clear). The \emph{closed neighborhood} of $v$ is $N_G[v] = \{v\} \cup N_G(v)$ (or simply $N[v]$ when the context is clear). The \emph{degree} of $v$ is defined as $\deg_G(v) = |N_G(v)|$. A graph $G$ is $k$-regular if every vertex has degree $k$. Specifically, we say $G$ is cubic if it is 3-regular. Given a pair $u, v\in V(G)$, the \emph{distance} between $u$ and $v$, denoted ${\rm dist}(u,v)$, is the minimum length among all paths from $u$ to $v$. A set $M\subseteq E(G)$ is a \emph{matching} in $G$ if no pair of edges in $M$ share a common vertex. An edge $e \in E(G)$ is referred to as \emph{bridge} if $G$ is connected and $G-e$ is disconnected. $G$ is said to be \emph{2-edge-connected} if it contains no bridges. The graph $K_{1, 3}$ is referred to as a \emph{claw}, and $G$ is said to be \emph{claw-free} if there is no induced subgraph isomorphic to $K_{1, 3}$. A set $A \subseteq V(G)$ is an \emph{i-packing} if for each pair $u,v\in A$, ${\rm dist}(u,v) > i$. A $1$-packing is more  traditionally called an independent set. A $(1, 1, 2, 3)$-packing coloring is a partition of $V(G)=A\cup B \cup C \cup D$ such that $A$ is an independent set in $G$, $B$ is an independent set in $G$, $C$ is a 2-packing in $G$, and $D$ is a $3$-packing in $G$. $F$ is a \emph{2-factor} in $G$ if $F$ is a 2-regular subgraph of $G$. 

We refer to a {\it diamond} as the graph $D$ obtained from the complete graph $K_4$ by deleting an edge. Note that in what follows, a {\it string of diamonds} is defined to be a maximal sequence $D_1, D_2, \dots, D_k$ of diamonds (which are isomorphic to $K_4-e$) in which, for each $i \in [k-1]$, $D_i$ has a vertex adjacent to a vertex in $D_{i+1}$ and has exactly two vertices which have degree two. Furthermore, a {\it ring of diamonds} is defined to be any connected claw-free cubic graph in which every vertex is in a diamond. Given a diamond, we refer to the two vertices of degree $3$ as the \emph{internal vertices} of the diamond, and the remaining two vertices as the \emph{external vertices} of the diamond.

\begin{theorem}\cite{Oum}\label{thm:2factor} A graph $G$ is $2$-edge-connected claw-free cubic if and only if either
\begin{enumerate}
\item[(i)] $G \cong K_4$,
\item[(ii)] $G$ is a ring of diamonds, or
\item[(iii)] $G$ can be built from a $2$-edge-connected cubic multigraph $H$ by replacing some edges of $H$ with strings of diamonds and replacing each vertex of $H$ with a triangle.
\end{enumerate}
\end{theorem}

In this paper, we will consider $2$-edge-connected, claw-free cubic graphs. Note that Petersen \cite{Petersen} was the first to show that every $2$-edge-connected cubic multigraph has a $2$-factor. When a cubic graph $G$ contains a $2$-factor $\mathcal{C} = C_1 \cup \cdots \cup C_k$, we will write that $G$ has a $(\mathcal{C}, M)$-decomposition where $M$ is the perfect matching in $G$ containing all remaining edges in $G$ that are not on $C_i$ for any $i\in[k]$. 

We also utilize the following result due to Plesn\'{i}k.

\begin{theorem}\cite{Plesnik}\label{thm:Plesnik} Let $G$ be an $(r-1)$-edge-connected regular graph of degree $r>0$ where $|V(G)|$ is even and let $H$ be an arbitrary set of $r-1$ edges. The graph $G' = G - H$ has a $1$-factor.
\end{theorem}

Focusing on a $2$-edge-connected cubic graph $G$, the above says that given any two edges $\{e, f\} \subset E(G)$ there exists a $2$-factor that contains the edges $e$ and $f$. Alternatively, and the way that we will use this result in the next section, is given an $e \in E(G)$, there exists a perfect matching $M$ in $G$ containing $e$. Since $E(G) - M$ is a $2$-factor in $G$, we will also say that given any $e \in E(G)$, there exists a $(\mathcal{C}, M)$-decomposition of $G$ where $e \in M$. Putting this together with Theorem~\ref{thm:2factor}, if $G$ is a $2$-edge-connected, diamond-free, claw-free cubic graph of order at least 6, then $G$ can be built from a $2$-edge-connected cubic multigraph $H$ by replacing each vertex of $H$ with a triangle. Therefore, when we refer to a $(\mathcal{C}, M)$-decomposition of $G$, we may assume that when we write $\mathcal{C} = C_1 \cup \cdots \cup C_k$ and we enumerate the vertices of $C_i$ as $x_1^ix_2^i\dots x_{n_i}^1$ for $i\in [k]$, we can begin the indexing so that for each $j \equiv 2 \pmod{3}$, $x_{j-1}^ix_j^ix_{j+1}^i$ induces a triangle in $G$. 


\section{Main Result}
We will focus on diamond-free, claw-free cubic graphs throughout the paper. We first focus on when such a graph is $2$-edge-connected.

\begin{theorem}\label{thm:2connected} If $G$ is a $2$-edge-connected diamond-free, claw-free cubic graph, then $G$ is $(1, 1, 2, 3)$-packing colorable. Moreover,  for any fixed triangle $uvw$ in $G$ with $(\mathcal{C}, M)$-decomposition of $G$ where $uw \in M$, we can find a $(1,1,2,3)$-packing coloring such that  $v$ is assigned the color $3$ and $u$ and $w$ are assigned the colors $1A$ or $1B$.
\end{theorem}

\begin{proof} As $K_4$ is clearly $(1, 1, 2, 3)$-colorable, we may assume that $G$ has order at least 6. Fix the triangle $uvw$ and let $e$ be the edge in $G$ incident to $v$, but not incident to $u$ or $w$. By Theorem~\ref{thm:Plesnik}, there exists a $(\mathcal{C}, M)$-decomposition of $G$ where $e \in M$. We proceed to assign exactly one vertex in each triangle in $G$ the color $2$ or $3$. If we assign a vertex of a triangle the color 2, we refer to it as a 2-triangle. If we assign a vertex of a triangle the color 3, we refer to it as a 3-triangle. Furthermore, any vertex assigned the color 3 (resp., 2) is referred to as a 3-vertex (resp., 2-vertex). 

Note that since $G$ is obtained by taking a claw-free multigraph $H$ and replacing each vertex in $H$ with a triangle, we can view our $(\mathcal{C}, M)$-decomposition of $G$ as being obtained from a $(\mathcal{C}', M')$-decomposition of $H$ as follows. Write $\mathcal{C}' = C_1' \cup \cdots \cup C_k'$ where $C_i' = h_1^ih_2^i\dots h_{m_i}^i$ and define $C_i = u_1^iv_1^iw_1^iu_2^iv_2^iw_2^i\dots u_{m_i}^iv_{m_i}^iw_{m_i}^i$ such that triangle $u_j^iv_j^iw_j^i$ replaced vertex $h_j^i$ in $H$. Moreover, $u_j^iw_j^i\in M$, $\mathcal{C}= C_1 \cup \cdots \cup C_k$ and each edge $h_k^ih_{\ell}^j \in M'$ is replaced with $v_k^iv_{\ell}^j$. We may reindex if necessary so that $v = v_1^1$. See Figure~\ref{fig:thm1} for an example of how to label the vertices of $G$ based on a labeling in $H$. In addition, for the rest of the paper, we will refer to the left and right triangle neighbor of a triangle. Consider a triangle, $u_{j}^iv_j^iw_{j}^i$, which we will call $T_{j}^i$. The left triangle neighbor of $T_{j}^i$  is the triangle $u_{\ell}^kv_{\ell}^kw_{\ell}^k$ that is different than $T_j^i$ and $u_{j}^i$ is adjacent to $w_{\ell}^k$. The right triangle neighbor of $T_{j}^i$ is the triangle $u_{r}^sv_{r}^sw_{r}^s$ that is different than $T_j^i$ and $w_{j}^i$ is adjacent to $u_{r}^s$.

We associate with this $(\mathcal{C}, M)$-decomposition a graph $\widetilde{G}$ with vertices $g_1, \dots, g_k$ such that $g_ig_j \in E(\widetilde{G})$ if and only if there is an edge between a vertex of $C_i$ and a vertex of $C_j$ in $G$. Choose a spanning tree $T$ of $\widetilde{G}$ rooted at $g_1$. Define $A_i = \{g_j \in T: \text{dist}_T(g_1, g_j) = i\}$ for $i \ge 1$.  Reindexing if necessary, we label the vertices in $\widetilde{G}$ according to a BFS-labeling of the vertices in $T$ rooted at $g_1$. That is, if $|A_1| = r_1$, then $A_1 = \{g_2, \dots, g_{r_1+1}\}$, etc. We also reindex the cycles $C_2, \dots, C_k$ based on the indexing of vertices in $T$. We proceed with a two stage process for coloring the vertices in $G$.  Note that once one of the vertices on a triangle is assigned the color $2$ or $3$, then we move on to the next triangle, and no other vertex on that triangle is assigned the color $2$ or $3$.

\vskip5mm
\noindent\textbf{Stage 1}\vskip5mm
First, we color some of the vertices of $C_1$ as follows. Let $H_1$ be the subgraph of $H$ induced by those vertices on $C_1'$ that have all three neighbors on $C_1'$ together with $h_1^1$. Note that $H_1$ is subcubic. Choose a maximal independent set $I$ from $H_1$ that contains $h_1^1$. It follows that $H_1 - I$ is a graph with maximum degree two so that $H_1 - I$ is a disjoint union of isolates, paths, and cycles. Let $\widetilde{H_1}$ be the graph induced by the vertices in $G$ that are on $C_1$ and replaced a vertex in $H_1 - I$. Then assign colors to some of the vertices of $C_1$ in order as follows.
\begin{itemize}
\item For each triangle in $G$ of the form $u_j^1v_j^1w_j^1$ that replaced $h_j^1 \in I$, assign $v_j^1$ the color $3$ so that $u_j^1v_j^1w_j^1$ is a $3$-triangle.
\item For each component of $\widetilde{H_1}$ built from an isolate in $H_1-I$, assign each vertex of the form $v_j^1$ the color $2$ so that $u_j^1v_j^1w_j^1$ is a $2$-triangle. Note that based on the choice of $I$, $v_{j'}^1 \in N_G(v_j^1) - \{u_j^1, w_j^1\}$ has already received color $3$.
\item Consider the components of $\widetilde{H_1}$ built from a cycle in $H_1 - I$. We can view such a component $J$ as a subcubic graph where every triangle contains exactly one vertex with degree two in $J$. Consider the triangle $u_j^1v_j^1w_j^1$ in $J$, which we will call $T_j^1$. Assume $a \in \{u_j^1, v_j^1, w_j^1\}$ is the vertex of degree $2$ on $T_j^1$, and $a$ is adjacent to a vertex in another triangle $u_i^1v_i^1w_i^1$. We claim $u_i^1v_i^1w_i^1$ is a $3$-triangle. If not, then this would imply the vertex $h_j^1$ has degree three in $H_1-I$, meaning $I \cup \{h_j^1\}$ is a larger independent set in $H_1$, which is a contradiction. Then assign $a$ the color $2$, and  $u_j^1v_j^1w_j^1$ is a $2$-triangle. Note that from above if $a$ is $v_j^1$, then the third neighbor of $v_j^1$ in $G$ has already been assigned the color $3$. If $a$ is $u_j^1$, then the left triangle neighbor of $T_j^i$ is $u_{j - 1}^1v_{j - 1}^1w_{j - 1}^1$ and is a $3$-triangle. Thus, $v_{j-1}^1$ has been assigned the color $3$. Similarly, if $a$ is $w_j^1$, then the right triangle neighbor of $T_j^i$ is $u_{j + 1}^1v_{j + 1}^1w_{j + 1}^1$ or ($u_1^1v_1^1w_1^1$ if $j = m_1$) and is a $3$-triangle. Thus, $v_{j+1}^1$ (or $v_1^1$) has been assigned the color $3$.
\item All that remains are the components of $\widetilde{H_1}$ built from a path in $H_1 - I$. We can view such a component $J$ as a subcubic graph where every triangle contains at least one vertex with degree two in $J$. Consider the triangle $u_j^1v_j^1w_j^1$ in $J$, which we will call $T_j^1$. Assume $a \in \{u_j^1, v_j^1, w_j^1\}$ is a vertex of degree $2$ on $T_j^1$, and $a$ is adjacent to a vertex in another triangle $u_i^1v_i^1w_i^1$. As above,  $u_i^1v_i^1w_i^1$ is a $3$-triangle. If $v_j^1$ has degree two in $J$, then the third neighbor of $v_j^1$ in $G$ has already been assigned the color $3$ from above. Assign $v_j^1$ the color $2$. If $v_j^1$ does not have degree two in $J$, then $u_j^1$ or $w_j^1$ (or both) have degree $2$ in $J$. If $u_j^1$ has degree two in $J$, assign $u_j^1$ the color $2$. As above, the left triangle neighbor of $T_j^1$ is $u_{j - 1}^1v_{j - 1}^1w_{j - 1}^1$ and is a $3$-triangle, and $v_{j-1}^1$ has been assigned the color $3$. If $v_j^1$ and $u_j^1$ both have degree $3$ in $J$, then $w_j^1$ has degree $2$ in $J$. Assign $w_j^1$ the color $2$. Similarly, if $a$ is $w_j^1$, then the right triangle neighbor of $T_j^1$ is  $u_{j + 1}^1v_{j + 1}^1w_{j + 1}^1$ (or $u_1^1v_1^1w_1^1$ if $j = m_1$) and  is a $3$-triangle. Hence, $v_{j+1}^1$ (or $v_1^1$) has been assigned the color $3$. 
\end{itemize}

Note that for each $h_j^1 \in H_1$, each triangle $T_j^1$ in $G$ of the form $u_j^1v_j^1w_j^1$ that replaced $h_j^1$ has either been assigned the color $2$ or $3$. Since $I$ was chosen to be an independent set in $H_1$, no $3$-triangles are adjacent. Thus, no 3-vertex on $C_1$ is within distance three of another 3-vertex on $C_1$. Moreover, if $u_j^1$ is assigned the color 2, then from above the left triangle neighbor of $T_j^1$ is a $3$-triangle. Similarly, if $w_j^1$ is assigned the color 2, then the right triangle neighbor of $T_j^1$ is a $3$-triangle, and if $v_j^1$ is assigned the color $2$, then the other triangle it is adjacent to is also a $3$-triangle.

Next, for each $2 \le i \le k$, we color some of the vertices of $C_i$ in a very similar way, only we let $H_i$ be the subgraph of $H$ induced by those vertices on $C_i'$ that have all three neighbors on $C_i'$ (and do not include $h_1^i$). We pick a maximal independent set $I$ of $H_i$. All the rest of the coloring process from above is the same. Note that in the case when $u_1^i$ is assigned the color $2$, then the left triangle neighbor of $T_j^i$ is $u_{m_i}^iv_{m_i}^iw_{m_i}^i$ and is a $3$-triangle.

\vskip5mm
\noindent\textbf{Stage 2}
\vskip5mm
We now proceed to finish coloring the vertices of $C_i$ for each $i\in [k]$. We begin with $C_1$. Let $T_{\alpha_1}^1, \dots , T_{\alpha_{t_1}}^1$ be those triangles with vertices on $C_1$ that are not yet 2- or 3-triangles such that if $i<j$ and $u_{\alpha_i}^1v_{\alpha_i}^1w_{\alpha_i}^1$ and $u_{\alpha_j}^1v_{\alpha_j}^1w_{\alpha_j}^1$ are the vertices of $T_{\alpha_i}^1$ and $T_{\alpha_j}^1$, respectively, then $\alpha_i<\alpha_j$. Note that $\alpha_1>1$ as we have already assigned $v_1^1$ the color $3$. 

For each $i\in [t_i]$:
\begin{itemize}
\item If the left triangle neighbor of $T_{\alpha_i}^1$ is a 3-triangle, then assign $u_{\alpha_i}^1$ the color 2.
\item else, if the right triangle neighbor of $T_{\alpha_i}^1$ is a 3-triangle, assign $w_{\alpha_i}^1$ the color 2, 
\item else, assign $v_{\alpha_i}^1$ the color $3$. 
\end{itemize}

Next, we finish coloring the vertices of each $C_i$ for $2 \le i \le k$. For $i\in \{2, \dots, k\}$:
\begin{itemize}
\item For each $v_j^i$ that is not on a 2- or 3-triangle yet and has a neighbor $v_{j'}^s$ not on $C_i$ that has been assigned the color 3, assign $v_j^i$ the color 2. 
\item Let $T_{\alpha_1}^i, \dots, T_{\alpha_{t_i}}^i$ be those triangles with vertices on $C_i$ that are not yet a 2- or 3-triangle. For each $j\in [t_i]$:
\begin{itemize}
\item If the left triangle neighbor of $T_{\alpha_j}^i$ is a 3-triangle, then assign $u_{\alpha_j}^i$ the color 2, 
\item else, if the right triangle neighbor of $T_{\alpha_j}^i$ is a 3-triangle, then assign $w_{\alpha_j}^i$ the color 2, 
\item else, assign $v_{\alpha_j}^i$ the color 3. 
\end{itemize}
\end{itemize}

Note that in all stages and cases, the following holds.
\begin{itemize}
\item The only vertices assigned the color $3$ are of the form $v_j^i$, and $v_j^i$ is only assigned the color $3$ if none of the triangles neighboring $T_j^i$ are $3$-triangles.
\item If $u_{j}^i$ is colored $2$, then the left triangle neighbor of $T_j^i$ is a $3$-triangle.
\item If $w_{j}^i$ is colored $2$, then the right triangle neighbor of $T_j^i$  is a $3$-triangle.
\item If $v_j^i$ is colored $2$, then the third neighbor of $v_j^i$ not on the triangle $T_j^i$ is colored $3$. Thus, $v_j^i$ is adjacent to a $3$-triangle.
\end{itemize}

Since the graph is finite, this process will eventually terminate, and at this point exactly one vertex from each triangle is assigned the color $2$ or $3$. The remaining vertices that are not assigned a color form an induced disjoint union of even length paths and even length cycles in $G$, and we can appropriately assign these vertices the colors $1A$ and $1B$.

We claim that this algorithm produces a $(1, 1, 2, 3)$-coloring. To see this, note that we only need to show that after Stage 2, the 2-vertices are a 2-packing in $G$ and the 3-vertices are a 3-packing in $G$ since the vertices colored $1A$ and $1B$ induce a bipartite graph in $G$.

First, consider the $3$-vertices. As mentioned above, the only vertices assigned the color $3$ are the form $v_j^i$, and $v_j^i$ is only assigned the color $3$ if none of the triangles neighboring $T_j^i$ are $3$-triangles. Thus, any two vertices that are color $3$ are distance at least $4$ apart, and the $3$-vertices are a $3$-packing.

Next, consider the $2$-vertices. Let $T_j^i$ be a 2-triangle with left triangle neighbor $T_{j-1}^i$ and right triangle neighbor $T_{j+1}^i$ (where $j-1$ and $j+1$ are taken to mean modulo $m_i$). In addition, we will assume $v_j^i$ is adjacent to a vertex on the triangle $T_r^s$. Note from above if a vertex is assigned the color $2$, then the other triangle it is adjacent to is a $3$-triangle. In addition, every triangle is assigned exactly one vertex of color $2$ or $3$.

First, assume $v_j^i$ is a $2$-vertex. Then $v_j^i$ was assigned the color $2$ because  $v_r^s$ and has color 3. Since each triangle has exactly one vertex of color 2 or 3, there is no 2-vertex on $C_s$ within distance two of $v_j^i$. Moreover, $w_{j-1}^i$ is not assigned color 2 as this would only have occurred in Stage 1 or Stage 2 where it was assumed $T_j^i$ was a 3-triangle. Similarly, $u_{j+1}^i$ is not a 2-vertex.

Next, assume $u_j^i$ is a 2-vertex. If $u_j^i$ is assigned the color 2 in Stage 1, then $v_{j-1}^i$ is color 3 based on our above discussion about $\widetilde{H_i}$, and $u_{j-1}^i$ and $w_{j-1}^i$ have been assigned 1A or 1B. Furthermore, $v_r^s$ is not assigned color 2 nor is $u_{j+1}^i$. So assume $u_j^i$ is assigned color 2 in Stage 2. Then $v_{j-1}^i$ is a 3-vertex, and $u_{j-1}^i$ and $w_{j-1}^i$ have been assigned 1A or 1B. If $j\ne m_i$, then either (1) $T_{j+1}^i$ is a 3-triangle so that $u_{j+1}^i$ is assigned 1A or 1B, or (2) $T_{j+1}^i$ became a 2-triangle before $T_j^i$ was assigned the color 2 and this only occurs when $v_{j+1}^i$ is assigned the color 2. On the other hand, if $j = m_i$, then either (1) $T_1^i$ is a 3-triangle so that $u_1^i$ is assigned 1A or 1B, or (2) $T_1^1$ is a 2-triangle and since $T_{m_i}^i$ was not yet a 2- or 3-triangle, then either $v_1^i$ or $w_1^i$ is assigned color 2. In either case, no 2-vertex in $G$ is within distance two of $u_j^i$. A similar argument shows no 2-vertex of the form $w_j^i$ is within distance two of another 2-vertex in $G$.

\end{proof}

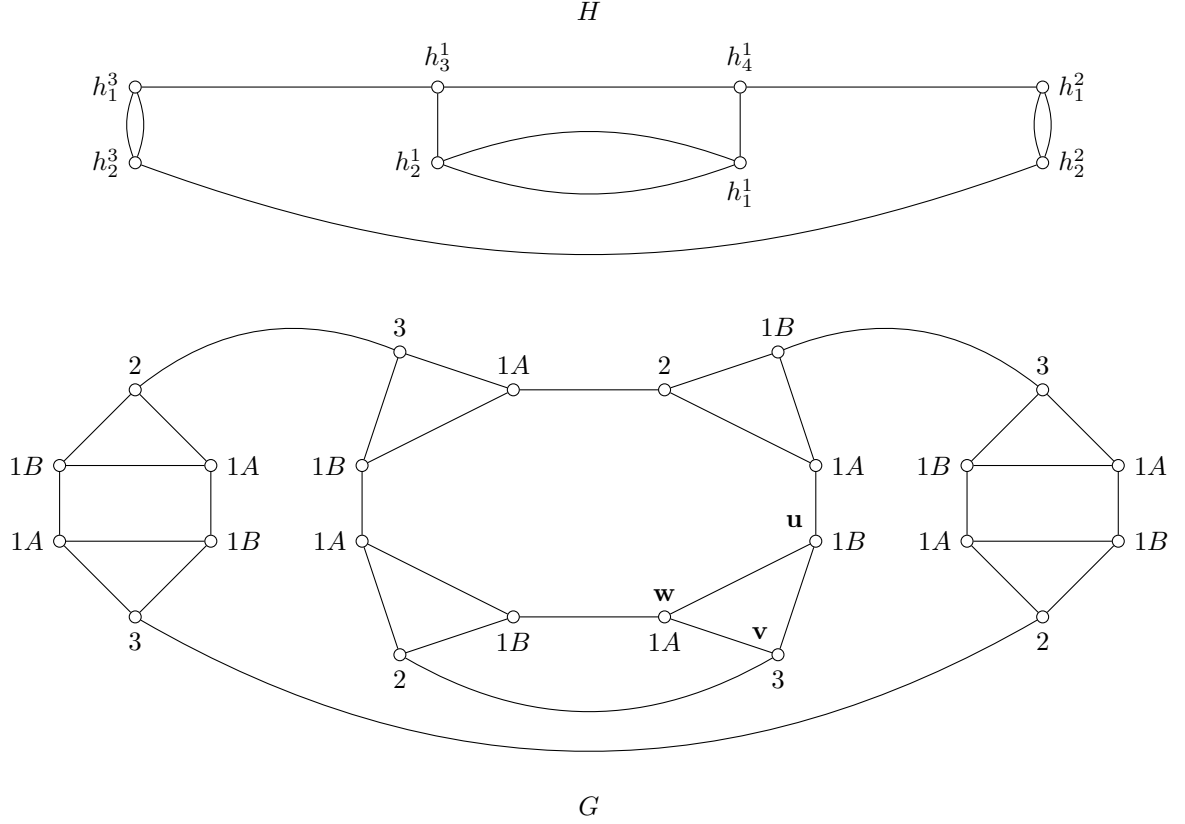
\begin{figure}[h]
\begin{center}
\begin{tikzpicture}[]
\tikzstyle{vertex}=[circle, draw, inner sep=0pt, minimum size=6pt]
\tikzset{vertexStyle/.append style={rectangle}}
	\vertex (1) at (-9,0) [scale=.75, label=left:$1A$] {};
	\vertex (2) at (-9,1) [ scale=.75, label=left:$1B$] {};
	\vertex (3) at (-8, 2) [ scale=.75, label=above:$2$] {};
	\vertex (4) at (-7, 1) [scale=.75, label=right:$1A$] {};
	\vertex (5) at (-7, 0) [scale=.75, label=right:$1B$] {};
	\vertex (6) at (-8, -1) [scale=.75, label=below:$3$] {};
	\vertex (7) at (3, 0) [scale=.75, label=left:$1A$] {};
	\vertex (8) at (3, 1) [scale=.75, label=left:$1B$] {};
	\vertex (9) at (4, 2) [scale=.75, label=above:$3$] {};
	\vertex (10) at (5, 1) [scale=.75, label=right:$1A$] {};
	\vertex (11) at (5,0) [scale=.75, label=right:$1B$] {};
	\vertex (12) at (4, -1) [scale=.75, label=below:$2$] {};
	\vertex(13) at (-5, 0) [scale=.75, label=left:$1A$] {};
	\vertex(14) at (-3, -1) [scale=.75, label=below:$1B$] {};
	\vertex(15) at (-4.5, -1.5) [scale=.75, label=below:$2$] {};
	\vertex(16) at (-1, -1) [scale=.75, label = above: $\mathbf{w}$, label=below:$1A$] {};
	\vertex(17) at (1, 0) [scale=.75, label = 125: $\mathbf{u}$, label= right:$1B$] {};
	\vertex(18) at (.5, -1.5) [scale=.75, label = 93: $\mathbf{v}$, label=below:$3$] {};
	\vertex(19) at (-1, 2) [scale=.75, label=above:$2$] {};
	\vertex(20) at (1, 1) [scale=.75, label=right:$1A$] {};
	\vertex(21) at (.5, 2.5) [scale=.75, label=above:$1B$] {};
	\vertex(22) at (-5, 1) [scale=.75, label=left:$1B$] {};
	\vertex(23) at (-3, 2) [scale=.75, label=above:$1A$] {};
	\vertex(24) at (-4.5, 2.5) [scale=.75, label=above:$3$] {};
	
	\vertex(25) at (-8, 6) [scale=.75, label=left:$h_1^3$] {};
	\vertex(26) at (-8, 5) [scale=.75, label=left:$h_2^3$] {};
	\vertex(27) at (-4, 6) [scale=.75, label=above:$h_3^1$] {};
	\vertex(28) at (0, 6) [scale=.75, label=above:$h_4^1$] {};
	\vertex(29) at (0,5) [scale=.75, label=below:$h_1^1$] {};
	\vertex(30) at (-4,5) [scale=.75, label=left:$h_2^1$] {};
	\vertex(31) at (4, 6) [scale=.75, label=right:$h_1^2$] {};
	\vertex(32) at (4,5) [scale=.75, label=right:$h_2^2$] {};
	\node(A) at (-2, 7) []{$H$};
	\node(B) at (-2, -3.5)[]{$G$};

	\path
		(1) edge (2)
		(2) edge (3)
		(3) edge (4)
		(4) edge (5)
		(5) edge (6)
		(1) edge (6)
		(2) edge (4)
		(1) edge (5)
		(7) edge (8)
		(8) edge (9)
		(9) edge (10)
		(10) edge (11)
		(11) edge (12)
		(7) edge (12)
		(6) edge[bend right] (12)
		(7) edge (11)
		(8) edge (10)
		(13) edge (14)
		(13) edge (15)
		(14) edge (15)
		(14) edge (16)
		(16) edge (17)
		(17) edge (18)
		(16) edge (18)
		(19) edge (20)
		(20) edge (21)
		(19) edge (21)
		(17) edge (20)
		(22) edge (23)
		(23) edge (24)
		(22) edge (24)
		(22) edge (13)
		(23) edge (19)
		(3) edge[bend left] (24)
		(9) edge[bend right] (21)
		(15) edge[bend right] (18)
		(25) edge[bend left=20] (26)
		(25) edge[bend right=20] (26)
		(25) edge (27)
		(27) edge (28)
		(28) edge (29)
		(29) edge[bend left=20] (30)
		(29) edge[bend right=20] (30)
		(27) edge (30)
		(28) edge (31)
		(31) edge[bend left=20] (32)
		(31) edge[bend right=20] (32)
		(26) edge[bend right=20] (32)

	;
\end{tikzpicture}
\end{center}
\caption{An example of using Theorem~\ref{thm:2connected} to assign a $(1, 1, 2, 3)$-packing coloring to a $2$-edge-connected diamond-free, claw-free cubic graph. Note that $v$ has been assigned the color $3$.}
\label{fig:thm1}
\end{figure}

Figure~\ref{fig:thm1} provides an example of how the algorithm given in Theorem~\ref{thm:2connected} assigns colors to the vertices. We will now use the coloring from Theorem~\ref{thm:2connected} to give a coloring that requires a specific vertex be a certain color in the next two results. 

\begin{corollary}\label{cor:onedegree3} Let $G$ be a $2$-edge-connected diamond-free, claw-free cubic graph. If $uvw$ is a triangle in $G$, then there exists a $(1,1,2,3)$-packing coloring that assigns each of $w$ and $v$ a color from $\{1A, 1B\}$ and assigns $u$ the color 2.  In addition, there is at most one vertex within distance two from $v$ that has color $3$, and this $3$-vertex is the neighbor of $v$ that is not on the triangle $uvw$.  

\end{corollary}

\begin{proof}
Consider the triangle $uvw$. Let $e$ be the edge in $G$ incident to $v$, but not incident to $u$ or $w$. By Theorem~\ref{thm:Plesnik}, there exists a $(\mathcal{C}, M)$-decomposition of $G$ where $e \in M$. We may assume $\mathcal{C} = C_1\cup \cdots \cup C_k$, where $uvw$ is a triangle on $C_1$. We can enumerate the vertices of $C_1$ as $x_1^1x_2^1\dots x_{n_1}^1$. We can then use the algorithm from Theorem~\ref{thm:2connected} where we relabel the vertices of $C_1'$ in the following way. We set $u = x_4^1$, $v=x_5^1$, $w=x_6^1$ and relabel the remaining vertices of $C_1$ accordingly. By the algorithm from Theorem~\ref{thm:2connected}, $x_1^1x_2^1x_3^1$ is a $3$-triangle and represents the vertex $h_1^1$, meaning $x_2^1$ is assigned the color 3. If the algorithm from Theorem~\ref{thm:2connected} assigns $u$ the color $2$, then we are done. Thus, assume the algorithm from Theorem~\ref{thm:2connected} does not assign $u$ the color $2$. As the left triangle neighbor of $uvw$ is a 3-triangle, it means that $v$ has been assigned the color 2 in Stage 1. In this case, we may simply swap the colors assigned to $u$ and $v$ and the resulting coloring is indeed a $(1, 1, 2, 3)$-packing coloring of $G$. Furthermore, since the only vertices that are assigned the color 3 are of the form $x_j^1$ where $j \equiv 2\pmod{3}$, it follows that any vertex within distance two from $v$ that has color 3 must be on the triangle that contains  the neighbor of $v$ other than $u$ or $w$.

\end{proof}

\begin{figure}[h]
\begin{center}
\begin{tikzpicture}[]
\tikzstyle{vertex}=[circle, draw, inner sep=0pt, minimum size=6pt]
\tikzset{vertexStyle/.append style={rectangle}}
	\vertex (1) at (-9,0) [scale=.75, label=left:$1A$] {};
	\vertex (2) at (-9,1) [ scale=.75, label=left:$1B$] {};
	\vertex (3) at (-8, 2) [ scale=.75, label=above:$3$] {};
	\vertex (4) at (-7, 1) [scale=.75, label=right:$1A$] {};
	\vertex (5) at (-7, 0) [scale=.75, label=right:$1B$] {};
	\vertex (6) at (-8, -1) [scale=.75, label=below:$2$] {};
	\vertex (7) at (3, 0) [scale=.75, label=left:$1A$] {};
	\vertex (8) at (3, 1) [scale=.75, label=left:$1B$] {};
	\vertex (9) at (4, 2) [scale=.75, label=above:$2$] {};
	\vertex (10) at (5, 1) [scale=.75, label=right:$1A$] {};
	\vertex (11) at (5,0) [scale=.75, label=right:$1B$] {};
	\vertex (12) at (4, -1) [scale=.75, label=below:$3$] {};
	\vertex(13) at (-5, 0) [scale=.75, label=left:$1B$] {};
	\vertex(14) at (-3, -1) [scale=.75, label=below:$1A$] {};
	\vertex(15) at (-4.5, -1.5) [scale=.75, label=below:$3$] {};
	\vertex(16) at (-1, -1) [scale=.75, label = above: $\mathbf{w}$, label=below:$1B$] {};
	\vertex(17) at (1, 0) [scale=.75, label = 125: $\mathbf{u}$, label= 275:\cancel{$1A$} $2$] {};
	\vertex(18) at (.5, -1.5) [scale=.75, label = 93: $\mathbf{v}$, label=below:\cancel{$2$} $1A$] {};
	\vertex(19) at (-1, 2) [scale=.75, label=below:$1A$] {};
	\vertex(20) at (1, 1) [scale=.75, label=right:$1B$] {};
	\vertex(21) at (.5, 2.5) [scale=.75, label=above:$3$] {};
	\vertex(22) at (-5, 1) [scale=.75, label=left:$2$] {};
	\vertex(23) at (-3, 2) [scale=.75, label=below:$1B$] {};
	\vertex(24) at (-4.5, 2.5) [scale=.75, label=above:$1A$] {};
	
		\vertex(25) at (-8, 6) [scale=.75, label=left:$h_1^3$] {};
	\vertex(26) at (-8, 5) [scale=.75, label=left:$h_2^3$] {};
	\vertex(27) at (-4, 6) [scale=.75, label=above:$h_4^1$] {};
	\vertex(28) at (0, 6) [scale=.75, label=above:$h_1^1$] {};
	\vertex(29) at (0,5) [scale=.75, label=below:$h_2^1$] {};
	\vertex(30) at (-4,5) [scale=.75, label=left:$h_3^1$] {};
	\vertex(31) at (4, 6) [scale=.75, label=right:$h_1^2$] {};
	\vertex(32) at (4,5) [scale=.75, label=right:$h_2^2$] {};
		\node(A) at (-2, 7) []{$H$};
	\node(B) at (-2, -3.5)[]{$G$};

	\path
		(1) edge (2)
		(2) edge (3)
		(3) edge (4)
		(4) edge (5)
		(5) edge (6)
		(1) edge (6)
		(2) edge (4)
		(1) edge (5)
		(7) edge (8)
		(8) edge (9)
		(9) edge (10)
		(10) edge (11)
		(11) edge (12)
		(7) edge (12)
		(6) edge[bend right] (12)
		(7) edge (11)
		(8) edge (10)
		(13) edge (14)
		(13) edge (15)
		(14) edge (15)
		(14) edge (16)
		(16) edge (17)
		(17) edge (18)
		(16) edge (18)
		(19) edge (20)
		(20) edge (21)
		(19) edge (21)
		(17) edge (20)
		(22) edge (23)
		(23) edge (24)
		(22) edge (24)
		(22) edge (13)
		(23) edge (19)
		(3) edge[bend left] (24)
		(9) edge[bend right] (21)
		(15) edge[bend right] (18)
				(25) edge[bend left=20] (26)
		(25) edge[bend right=20] (26)
		(25) edge (27)
		(27) edge (28)
		(28) edge (29)
		(29) edge[bend left=20] (30)
		(29) edge[bend right=20] (30)
		(27) edge (30)
		(28) edge (31)
		(31) edge[bend left=20] (32)
		(31) edge[bend right=20] (32)
		(26) edge[bend right=20] (32)

	;
\end{tikzpicture}
\end{center}
\caption{An example of using Corollary \ref{cor:onedegree3} to color a graph $G$ that is a $2$-edge-connected diamond-free, claw-free cubic graph where $v$ has color $1A$, $u$ has color $2$, and $w$ has color $1B$.}
\label{fig:cor1}
\end{figure}

Figure \ref{fig:cor1} shows how Corollary \ref{cor:onedegree3} uses the algorithm in Theorem \ref{thm:2connected} to create a coloring such that $v$ is colored $1A$, $u$ is colored $2$, $w$ is colored $1B$, and there is at most one vertex within distance two from $v$ colored $3$.  Corollary \ref{cor:onedegree2} utilizes the algorithm in Theorem \ref{thm:2connected} with a particular type of $2$-edge-connected diamond-free, claw-free subcubic graph.

\begin{corollary}\label{cor:onedegree2} If $G$ is a $2$-edge-connected diamond-free, claw-free subcubic graph with $\delta(G)\ge 2$ and exactly one vertex with degree $2$, call it $v$, then $G$ is $(1, 1, 2, 3)$-packing colorable. Moreover, $v$ is on a triangle $uvw$ in $G$, and 
\begin{itemize}
\item there exists a $(1,1,2,3)$-packing coloring such that $v$ is assigned either $1A$ or $1B$ and no vertex within distance two from $v$ has color $3$, and 
\item there exists a $(1, 1, 2, 3)$-packing coloring such that $v$ is assigned the color $3$ and $u$ and $w$ are assigned the colors $1A$ or $1B$.
\end{itemize}

\end{corollary}

\begin{proof} Note that since $G$ is diamond-free, there exists $a \in N_G(u)-N_G(w)$ and $b \in N_G(w) - N_G(u)$. We claim that $ab\not\in E(G)$. For the sake of contradiction, assume $ab \in E(G)$. Since $G$ is claw-free, there is a vertex $r$ adjacent to both $a$ and $b$. However, the third neighbor of $r$ different from $a$ and $b$ would then be a bridge of $G$, which does not exist in $G$. Thus, $ab\not\in E(G)$. Therefore, we can create the graph $\widetilde{G}$ from $G$ by removing the vertices $\{u, v, w\}$ and adding the edge $ab$. Thus, $\widetilde{G}$ is a 2-edge-connected, claw-free, diamond-free cubic graph and we can find a $(\mathcal{C}, M)$-decomposition for $\widetilde{G}$ where $ab\not\in M$. If we assume $\mathcal{C} = C_1\cup \cdots \cup C_k$ and $ab$ is an edge on $C_1$, it follows that if we enumerate the vertices of $C_1$ as $x_1^1x_2^1\dots x_{n_1}^1$, then $a=x_j^1$ and $b=x_{j+1}^1$ where $j \equiv 0 \pmod{3}$ for $a$ is on a triangle in $\widetilde{G}$ that does not contain $b$. Therefore, $(\mathcal{C}', M')$ defined by $\mathcal{C}' = C_1' \cup C_2\cup \cdots \cup C_k$ with $C_1' = uvwx_{j+1}^1x_{j+2}^1\dots x_{n_i}^1x_1^1\dots x_j^1$ and $M' = M \cup \{uw\}$ is a decomposition for ${G}$. We can then use the algorithm from Theorem~\ref{thm:2connected} where we relabel the vertices of $C_1'$ in the following ways to achieve each case. Note that when we use this algorithm that $v$ will never be colored $2$ since it does not have a third neighbor. 
\begin{itemize}
\item If  we want $v$ to receive color $1A$ or $1B$, then we set $u = x_4^1$, $v=x_5^1$, $w=x_6^1$ and add three to the subscript for each remaining vertex on $C_1'$. Since $u = x_4^1$, it must be the case that $a = x_3^1$. By the algorithm from Theorem~\ref{thm:2connected}, $x_1^1x_2^1x_3^1$ is a $3$-triangle. Since a $3$-triangle is never adjacent to another $3$-triangle, $uvw$ must a $2$-triangle where $u$ is a 2-vertex as  the left triangle of $uvw$ is a $3$-triangle. Furthermore, the algorithm from Theorem~\ref{thm:2connected} only assigns one vertex in each triangle the color $2$ or $3$, implying $v$ and $w$ are assigned the colors $1A$ or $1B$. Also recall that the algorithm from Theorem~\ref{thm:2connected} only assigns vertices of the form $v_j^i$ the color $3$. Hence, no vertex within distance two from $v$ has color $3$ since $v$ has degree $2$. Thus, $v$ is assigned either $1A$ or $1B$ and no vertex within distance two from $v$ has color $3$.
\item If we want $v$ to receive color $3$, then we set $u=x_1^1$, $v=x_2^1$, $w=x_3^1$, and add three to the subscript for each remaining vertex on $C_1'$. Then $v$ will be colored $3$ by the algorithm from Theorem~\ref{thm:2connected}. Note that $uvw$ are on the same triangle. Again, since the algorithm from Theorem~\ref{thm:2connected} only assigns one vertex in each triangle the color $2$ or $3$, $v$ will be assigned the color $3$ and $u$ and $w$ will be assigned the colors $1A$ or $1B$.
\end{itemize}
\end{proof}

\begin{figure}[h]
\begin{center}
\begin{tikzpicture}[]
\tikzstyle{vertex}=[circle, draw, inner sep=0pt, minimum size=6pt]
\tikzset{vertexStyle/.append style={rectangle}}
	
	\vertex (7) at (5, 1) [scale=.75, label = above: $\mathbf{v}$, label=right:$1A$] {};
	\vertex (8) at (3, 1) [scale=.75, label = above: $\mathbf{w}$, label=left:$1B$] {};
	\vertex (9) at (4, 2) [scale=.75, label = above: $\mathbf{u}$, label=right:$2$] {};

	\vertex(14) at (-3, 0) [scale=.75, label=left:$2$] {};
	\vertex(16) at (-1, 0) [scale=.75, label=right:$1B$] {};
	\vertex(17) at (-2, -1) [scale=.75,  label= below: $1A$] {};
	\vertex(19) at (-3, 1) [scale=.75, label=left:$3$] {};
	\vertex(20) at (-1, 1) [scale=.75, label=right:$1A$] {};
	\vertex(21) at (-2, 2) [scale=.75, label=above:$1B$] {};

	\path

		(7) edge (8)
		(8) edge (9)
		(7) edge (9)

		(19) edge (20)
		(20) edge (21)
		(19) edge (21)
		(8) edge[bend left] (17)
		(14) edge (17)
		(16) edge (17)
		(14) edge (16)
		(19) edge (14)
		(20) edge (16)
		
		(21) edge[bend left] (9)
		
	;
\end{tikzpicture}
\end{center}
\caption{An example using Corollary~\ref{cor:onedegree2} to color a graph $G$ that is a $2$-edge-connected diamond-free, claw-free subcubic graph with $\delta(G)\ge 2$ and exactly one vertex with degree $2$ that is colored $1A$ and is not within distance two from a color $3$ vertex.}
\label{fig:cor2}
\end{figure}
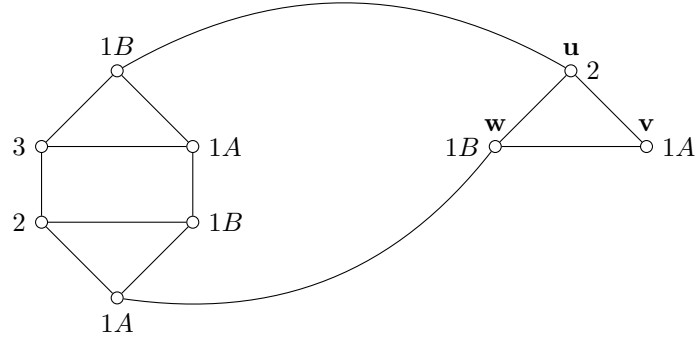

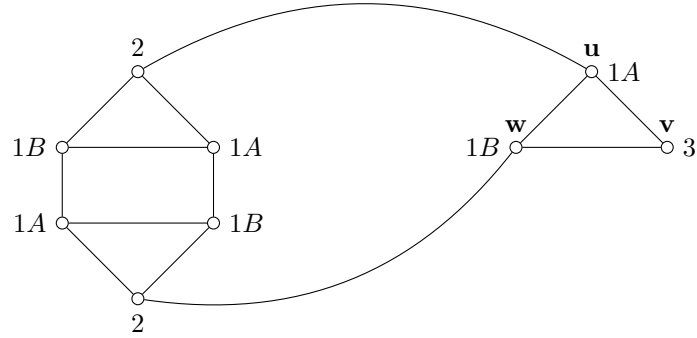
\begin{figure}[h]
\begin{center}
\begin{tikzpicture}[]
\tikzstyle{vertex}=[circle, draw, inner sep=0pt, minimum size=6pt]
\tikzset{vertexStyle/.append style={rectangle}}
	
	\vertex (7) at (5, 1) [scale=.75, label = above: $\mathbf{v}$, label=right:$3$] {};
	\vertex (8) at (3, 1) [scale=.75, label = above: $\mathbf{w}$, label=left:$1B$] {};
	\vertex (9) at (4, 2) [scale=.75, label = above: $\mathbf{u}$, label=right:$1A$] {};

	\vertex(14) at (-3, 0) [scale=.75, label=left:$1A$] {};
	\vertex(16) at (-1, 0) [scale=.75, label=right:$1B$] {};
	\vertex(17) at (-2, -1) [scale=.75,  label= below: $2$] {};
	\vertex(19) at (-3, 1) [scale=.75, label=left:$1B$] {};
	\vertex(20) at (-1, 1) [scale=.75, label=right:$1A$] {};
	\vertex(21) at (-2, 2) [scale=.75, label=above:$2$] {};

	\path

		(7) edge (8)
		(8) edge (9)
		(7) edge (9)

		(19) edge (20)
		(20) edge (21)
		(19) edge (21)
		(8) edge[bend left] (17)
		(14) edge (17)
		(16) edge (17)
		(14) edge (16)
		(19) edge (14)
		(20) edge (16)
		
		(21) edge[bend left] (9)
		
	;
\end{tikzpicture}
\end{center}
\caption{An example using Corollary~ \ref{cor:onedegree2} to color a graph $G$ that is a $2$-edge-connected diamond-free, claw-free subcubic graph with $\delta(G)\ge 2$ and exactly one vertex with degree $2$ that is colored $3$.}
\label{fig:cor3}
\end{figure}

Figure~\ref{fig:cor2} gives an example of using Corollary~\ref{cor:onedegree2} to color a graph $G$ that is $2$-edge-connected diamond-free, claw-free subcubic graph with $\delta(G)\ge 2$ and exactly one vertex $v$ with degree $2$, where $v$ is assigned the color $1A$ and no vertex within distance two from $v$ has color $3$. Figure~\ref{fig:cor3} gives an example of using Corollary~\ref{cor:onedegree2}  to color a graph $G$ that is a $2$-edge-connected diamond-free, claw-free subcubic graph with $\delta(G)\ge 2$ and exactly one vertex $v$ with degree $2$, where $v$ is assigned the color $3$. The following statements are true for all of the colorings given in Theorem~\ref{thm:2connected}, Corollary~\ref{cor:onedegree3}, and Corollary~\ref{cor:onedegree2}.

\begin{itemize}
\item Only one vertex from each triangle is assigned the color $2$ or $3$. The other two vertices are assigned the colors $1A$ or $1B$.
\item The only vertices assigned the color $3$ are of the form $v_j^i$, and no $3$-triangle is adjacent to another $3$-triangle.
\item A vertex assigned the color 2 has a neighbor which is on a $3$-triangle. 
\end{itemize}

Next, we extend our $(1, 1, 2, 3)$-packing coloring to a diamond-free, claw-free cubic graph with bridges. Figure \ref{fig:thm2} gives an example of this $(1, 1, 2, 3)$-packing coloring.

\begin{theorem} If $G$ is a diamond-free, claw-free cubic graph, then $G$ is $(1, 1, 2, 3)$-packing colorable.
\label{thm:mainthm}
\end{theorem}

\begin{proof}
We may assume that $G$ is connected and contains bridges. We let $B(G)$ be the set of bridges, meaning that we can write $G- B(G) = G_1\cup \cdots \cup G_k$ where each $G_i$ is a 2-edge-connected, diamond-free, claw-free subcubic component. Note that since $G$ is claw-free, for every vertex $v \in G$, there exists a $G_i$ such that $v$ has at least two neighbors in $G_i$. Also, $|V(G_i)| \geq 3$ for all $i \in [k]$. We associate with $G$ a tree $T_G$ with vertex set $T_G = \{v_1, \dots, v_k\}$ such that $v_iv_j \in E(T_G)$ if and only if there is a bridge in $G$ between a vertex in $G_i$ and a vertex in $G_j$. We shall root $T_G$ at $v_1$ and refer to $v_j \in V(T_G)$ as being at level $i$ if ${\rm{dist}}_{T_G}(v_1, v_j) = i$. We will also say that $v_1$ is at level $0$. 

For each $i\in [k]$, we create a supergraph $\widetilde{G_i}$ of $G_i$ as follows. Let $\{r_1^i, \dots, r_{t_i}^i\}$  be the vertices of degree two in $G_i$. Then $r_j^i$ has two neighbors in $G_i$, and those neighbors are adjacent since $G$ is claw-free and $r_j^i$ is incident to a bridge. That is, $r_j^i$ is on a triangle in $G_i$. Set $\widetilde{G_i} = G_i$ if every vertex in $G_i$ has degree two (which can only happen if $G_i \cong K_3$). Otherwise, if $t_i$ is even, then we obtain $\widetilde{G_i}$ from $G_i$ by adding the edges $r_{\alpha}^ir_{\alpha+1}^i$ for odd $\alpha\in[t_i]$. If $t_i=1$, then we set $\widetilde{G_i} = G_i$. If $t_i \ge 3$ is odd, then we obtain $\widetilde{G_i}$ from $G_i$ by adding the edges $r_{\alpha}^ir_{\alpha+1}^i$ for even $\alpha\in [t_i]$. Note that in each case, $\widetilde{G_i}$ is a 2-edge-connected, claw-free, subcubic graph with at most one vertex with degree two. We claim that $\widetilde{G_i}$ is also diamond-free. Suppose to the contrary that $\widetilde{G_i}$ contains a diamond with vertices $r_{\alpha}^i, r_{\alpha+1}^i, x$, and $y$. From above, we may assume $r_{\alpha}^i$ is on a triangle with $x$ and $y$. Since $r_{\alpha}^i$ is also adjacent to $r_{\alpha + 1}^i$, then $r_{\alpha}^i$ must be an interior vertex of a diamond. Since $x$ and $y$ are also adjacent, $r_{\alpha+1}^i$ is an exterior vertex. Then, $r_{\alpha}^i$ and $r_{\alpha+1}^i$ share a common neighbor, say $x$. However, $r_{\alpha+1}^i$ must have another neighbor in $G_i$, say $z$. It follows that $z$ and $x$ in $G_i$ must be adjacent; otherwise, $N_G[r_{\alpha + 1}^i]$ is a claw. However, this implies $x$ has degree $4$, which is a contradiction. Thus, $\widetilde{G_i}$ is indeed diamond-free. 

We proceed to color the vertices $G_i$ for each $i\in [k]$ based on the level of its corresponding vertex $v_i$ in $T_G$. We start with $G_1$ (since $v_1$ is the only vertex at level $0$).  By Corollary~\ref{cor:onedegree3} and Corollary~\ref{cor:onedegree2}, there exists a $(1, 1, 2, 3)$-packing coloring of $\widetilde{G_1}$. It follows that this coloring of the vertices of $\widetilde{G_1}$ is also a $(1, 1, 2, 3)$-packing coloring of $G_1$. 
Now set the current level to $\ell=1$. For each vertex $v_i \in T_G$ at level $\ell$, we pick $r_1^i$ to be the unique vertex in $G_i$ with degree two that has a neighbor in $G_s$ where the corresponding vertex $v_s \in T_G$ has level $\ell-1$.  That is, when we enumerate the vertices of degree two in $G_i$ as $\{r_1^i, \dots, r_{t_i}^i\}$, we shall assume that $ r_1^i$ has a neighbor on a component of $G- B(G)$ ``one level up". Moreover, we set $v = r_1^i$.  We will also assume that the neighbor of $r_1^i$ not in $G_i$ is $z$, which has already received a color and is contained in $G_s$. We color the vertices of $\widetilde{G_i}$ depending on how we constructed $\widetilde{G_i}$.
\vskip5mm
\noindent\textbf{Case 1:} Assume $\widetilde{G_i} = G_i \cong K_3$.
\vskip2mm

Note that the $G_i = r_1^ir_2^ir_3^i$.  If $z$ has color $1A$, assign $r_1^i$ the color $1B$, $r_2^i$ color $1A$ and $r_3^i$ color $2$. Note that $r_3^i$ is distance at least three from any vertex in $G_s$ with color $2$. Similarly, if $z$ has color $1B$, then assign $r_1^i$ the color $1A$, $r_2^i$ the color $1B$, and $r_3^i$ the color $2$. If $z$ has color $3$, then assign $r_1^i$ the color $2$, $r_2^i$ the color $1B$, and $r_3^i$ the color $1A$. Finally, if $z$ has color $2$, then assign $r_1^i$ the color $3$, $r_2^i$ the color $1A$, and $r_3^i$ the color $1B$. Note that in each case, we have assigned  exactly one vertex a color from $\{2, 3\}$. Therefore, if $r_1^i$ is assigned 2 and $G_s\cong K_3$, then no vertex in $G_s$ is assigned the color $2$. On the other hand, if $r_1^i$ is assigned the color $3$ and $G_s \cong K_3$, then $z$ is assigned the color $2$ and this only occurs if $G_s = u^sv^sz$ where $v^s$ is adjacent to some vertex $p$ on $G_{s'}$ where $v_{s'}$ in $T_G$ is at level $\ell-2$ and $p$ is assigned color $1A$ or $1B$.

\vskip5mm
\noindent\textbf{Case 2:} Assume $\widetilde{G_i}\not\cong K_3$.
\vskip2mm
Let $u$ and $w$ be the other two vertices in $G_i$ that $r_1^i$ is adjacent to. Suppose first $z$ has color $1A$ or 3. By Corollary~\ref{cor:onedegree3} and Corollary~\ref{cor:onedegree2}, there is a $(1, 1, 2, 3)$-packing coloring $c_i$ of $\widetilde{G_i}$ where $r_1^i$ and $w$ receives the color $1A$ or $1B$, and $u$ receives the color $2$. If $r_1^i$ has degree two in $\widetilde{G_i}$, then no vertex within distance two of $r_1^i$ in $\widetilde{G_i}$ has color $3$ by Corollary~\ref{cor:onedegree2}. It follows that  no vertex in $G_i$ within distance two of $r_1^i$ has color $3$. Next, assume $r_1^i$ has degree $3$ in $\widetilde{G_i}$. By Corollary~\ref{cor:onedegree3}, there is at most one vertex within distance two of $r_1^i$ in $\widetilde{G_i}$ that has color $3$. In addition, by Corollary~\ref{cor:onedegree3}, if this $3$-vertex exists, then it is the third neighbor of $r_1^i$ in $\widetilde{G_i}$ outside of the triangle $ur_1^iw$. Then call this vertex $g'$. Based on the construction of $\widetilde{G_i}$, the edge $r_1^ig'$ was added to $G_i$ to construct $\widetilde{G_i}$. Thus, $g'$  is not within distance two of $r_1^i$ in $G_i$. Therefore, in both cases, no vertex in $G_i$ is within distance two of $r_1^i$ that has color $3$. If this coloring assigns $r_1^i$ the color $1A$, then simply recolor every vertex assigned $1A$ with $1B$ and every vertex with $1B$ the color $1A$. Note that any vertex that has received color $2$ in $G_i$ is distance at least three from a vertex in $G_s$ that has received the color $2$. Similarly, since there is no vertex in $G_i$ within distance two of $r_1^i$ that has color $3$, any vertex with color $3$ in $G_i$ has distance at least four from any vertex in $G_s$ that has color $3$. Note further that if we had assumed that $z$ has color $1B$, we could analogously find a coloring for $G_i$. 

Finally, suppose $z$ has color $2$. Since $G$ is claw-free, $z$ is on a triangle in $\widetilde{G_s}$ and $z$ has degree two in $G_s$. By Corollary~\ref{cor:onedegree3} and Corollary~\ref{cor:onedegree2}, there is a $(1, 1, 2, 3)$-coloring $c_i$ of $\widetilde{G_i}$ where $r_1^i$ receives the color 3 and the neighbors of $r_1^i$ in $\widetilde{G_i}$ receive color $1A$ or $1B$.

\vskip5mm
We shall assume now that the above coloring process is completed and all vertices of $G$ have received a color. It is clear at each stage that the vertices that receive color 1A (or 1B) is an independent set in $G$. It is also clear that for any pair $u,v \in V(G_i)$ with color 2 (resp., 3) that ${\rm dist}_G(u,v) \ge 3$ (resp.,  ${\rm dist}_G(u,v) \ge 4$). All that remains is to show that if $u \in G_i$ and $v \in G_j$ where $i\ne j$ receive color 2 (resp., 3), then ${\rm dist}_G(u,v) \ge 3$ (resp., ${\rm dist}_G(u,v) \ge 4$). 

Suppose first that $u$ and $v$ each receive color $2$, and  ${\rm dist}_G(u,v) \le2$. We may assume that either $u$ has degree two in $G_i$ or $v$ has degree two in $G_j$ for otherwise ${\rm dist}_G(u,v) \ge 3$. Without loss of generality, assume $u$ has degree two in $G_i$. This implies that we may assume that $v$ is on the triangle in $G_j$ containing a neighbor of $u$. Suppose first that the level of $v_i$ in $T_G$ is less than the level of $v_j$ in $T_G$. Thus, $r_1^j$ is adjacent to $u$ and $v$ is on the same triangle in $G_j$ as $r_1^j$.  Regardless of whether $G_j$ is isomorphic to $K_3$ or not, $r_1^j$ is assigned 3 making the triangle containing $v$ a 3-triangle, contradicting the assumption that $v$ has color 2. Therefore, we shall assume that the level of $v_i$ in $T_G$ is greater than the level of $v_j$ in $T_G$. It follows that $u = r_1^i$, and we may assume that $z$ is the neighbor of $u$ in $G_j$. This implies that $G_i \cong K_3$ for otherwise $u$ would not receive color 2. Moreover, $z$ was first assigned color 3 as the vertices of $G_j$ are colored before $G_i$. However, this implies that the triangle in $G_j$ containing $v$ is a 3-triangle, which again is a contradiction. 

Lastly, suppose that $u$ and $v$ each receive color 3, and ${\rm dist}_G(u,v) \le 3$. Let $T^u$ be the triangle containing $u$ in $G_i$ and $T^v$ be the triangle containing $v$ in $G_j$. Without loss of generality, we shall assume that the level of $v_i$ in $T_G$ is less than or equal to the level of $v_j$ in $T_G$. 

\begin{itemize}
\item Suppose that $G_i \cong K_3$. Thus, $G_i = r_1^ir_2^ir_3^i$ where $u=r_1^i$ (based on the coloring process in Case 1). It follows that $r_2^i$ and $r_3^i$ each receive a color from $\{1A, 1B\}$. Suppose first that $v_i$ and $v_j$ are at the same level in $T_G$. This means that $u$ is adjacent to some vertex in $G_s$, call it $z$,  and $v$ is adjacent to some vertex in $G_s$, call it $z'$, where the level of $v_s$ in $T_G$ is less than the level of $v_i$ or $v_j$. Note that $z\ne z'$ for otherwise $z$ is the center of claw. Moreover, $zz' \in E(G_s)$ for otherwise ${\rm dist} (u,v) \ge 4$. This implies that $G_s \cong K_3$. By Case 1, since $u$ is assigned color 3, it must be that $z$ was assigned color 2. Thus, $z'$ has a color from $\{1A, 1B\}$. However, by Case 1 or 2, $v$ should have been assigned a color from $\{1A, 1B\}$, which is a contradiction. Thus, we shall assume that the level of $v_i$ in $T_G$ is strictly less than the level of $v_j$ in $T_G$. This implies that $r_1^j$ is on $T^v$ and is adjacent to either $r_2^i$ or $r_3^i$.  Regardless of whether $G_j$ is isomorphic to $K_3$ or not, $r_1^j$ receives a color from $\{1A, 1B\}$ in Case 1 and 2. Furthermore, $T^v$ becomes a 2-triangle, contradicting that $v$ has color 3. 
\item Suppose  that $G_i \not\cong K_3$. Suppose first that there is a bridge between $G_i$ and $G_j$. If $T^u$ does not contain the vertex of $G_i$ incident to the bridge between $G_i$ and $G_j$, then ${\rm{dist}}(u,v)\ge 4$ which is a contradiction. Thus, we may assume that $T^u$ does contain the vertex of $G_i$ incident to the bridge between $G_i$ and $G_j$; in particular, $u$ is incident to the bridge between $G_i$ and $G_j$. Suppose first that $G_j \not\cong K_3$. It follows by Case 2 that $r_1^j$ receives a color from $\{1A, 1B\}$ and is not within distance two of another vertex in $G_j$ with color 3, contradicting that  ${\rm dist}_G(u,v) \le 3$. Therefore, we shall assume that $G_j \cong K_3$. By Case 1, $G_j$ becomes a 2-triangle, another contradiction. Thus, we shall assume that the level of $v_i$ in $T_G$ is $\ell$ and the level of $v_j$ in $T_G$ is $\ell + 2$. It follows that there exists $G_s$ containing vertices $z$ and $z'$ such that $uz \in E(G)$, $r_1^jz' \in E(G)$ and $v_s$ is at level $\ell +1 $ in $T_G$. Since ${\rm dist}_G(u,v) \le 3$, $G_s \cong K_3$ and $r_1^j = v$. By Case 1, $r_1^s$ is assigned the color 2 and each of $\{r_2^s, r_3^s\}$ are assigned a color from $\{1A, 1B\}$. Regardless of whether $G_j$ is isomorphic to $K_3$, by Case 1 and 2 $r_1^j$ is assigned a color from $\{1A, 1B\}$, another contradiction. 
\end{itemize}
Having exhausted all possibilities, we may conclude that we have a $(1, 1, 2, 3)$-packing coloring of $G$.

\end{proof}

\begin{figure}[h]
\begin{center}
\begin{tikzpicture}[]
\tikzstyle{vertex}=[circle, draw, inner sep=0pt, minimum size=6pt]
\tikzset{vertexStyle/.append style={rectangle}}
	\vertex(13) at (-5, 0) [scale=.75, label=left:$2$] {};
	\vertex(14) at (-3, -1) [scale=.75, label=below:$1B$] {};
	\vertex(15) at (-4.5, -1.5) [scale=.75, label=below:$1A$] {};
	\vertex(16) at (-1, -1) [scale=.75,  label=below:$1A$] {};
	\vertex(17) at (1, 0) [scale=.75, label= right: $2$] {};
	\vertex(18) at (.5, -1.5) [scale=.75, label=below:$1B$] {};
	\vertex(19) at (-3, 1) [scale=.75, label=below:$1B$] {};
	\vertex(20) at (-1, 1) [scale=.75, label=below:$1A$] {};
	\vertex(21) at (-2, 2) [scale=.75, label=above:$3$] {};

	\vertex(22) at (4, 0) [scale=.75, label=left:$1A$] {};
	\vertex(23) at (6, -1) [scale=.75, label=below:$1B$] {};
	\vertex(24) at (4.5, -1.5) [scale=.75, label=below:$3$] {};
	\vertex(25) at (8, -1) [scale=.75,  label=below:$1A$] {};
	\vertex(26) at (10, 0) [scale=.75, label= right: $1B$] {};
	\vertex(27) at (9.5, -1.5) [scale=.75, label=below:$2$] {};
	\vertex(28) at (6, 1) [scale=.75, label=below:$2$] {};
	\vertex(29) at (8, 1) [scale=.75, label=below:$1B$] {};
	\vertex(30) at (7, 2) [scale=.75, label=above:$1A$] {};

	\path
				(13) edge (14)
		(13) edge (15)
		(14) edge (15)
		(14) edge (16)
		(16) edge (17)
		(17) edge (18)
		(16) edge (18)
		(19) edge (20)
		(20) edge (21)
		(19) edge (21)
		(17) edge (20)
		(13) edge (19)
		(22) edge (23)
		(22) edge (24)
		(23) edge (24)
		(22) edge (28)
		(25) edge (26)
		(25) edge (27) 
		(26) edge (27)
		(26) edge (29)
		(28) edge (29)
		(28) edge (30)
		(29) edge (30)
		(23) edge (25)
	
		(21) edge[bend left] (30)
		(15) edge[bend right] (18)
		(24) edge[bend right] (27)

	;
\end{tikzpicture}
\end{center}
\caption{An example of using Theorem \ref{thm:mainthm} to assign a $(1, 1, 2, 3)$ coloring to a diamond-free, claw-free cubic graph.}
\label{fig:thm2}
\end{figure}
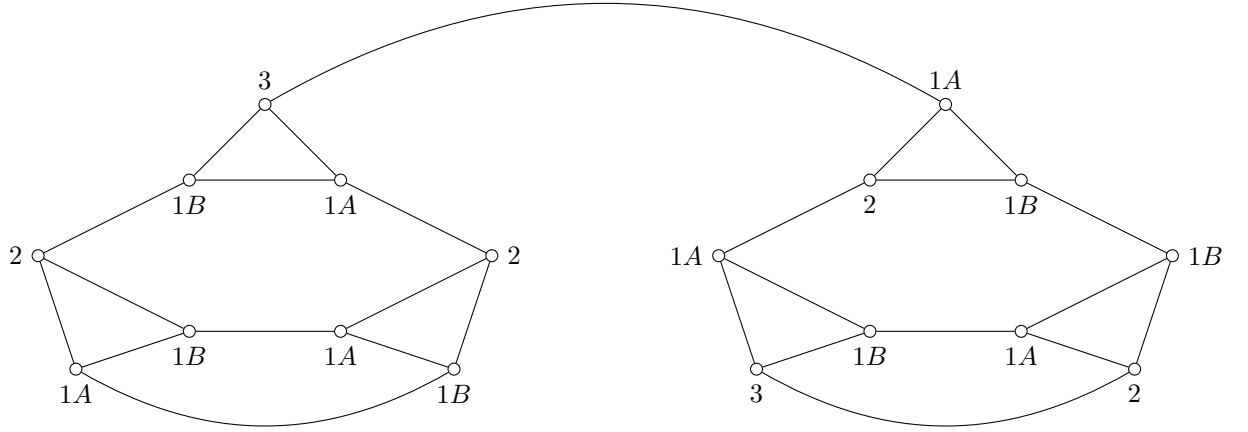

\section*{Acknowledgements}
K.K. would like to thank the financial support provided by an AMS-Simons Research Enhancement Grant for Primarily Undergraduate Institution Faculty.

\end{document}